\documentclass[english]{amsart}
\usepackage{amsmath} % do we need this?
\usepackage{amssymb} % do we need this?
\title{Playing simple loony dots and boxes endgames optimally.}
\author{Kevin Buzzard} \email{k.buzzard@imperial.ac.uk} \address{Department of Mathematics, Imperial College London, 180 Queen's Gate, London SW7 2AZ, England.}
\author{Michael Ciere} \email{} \address{}
\usepackage{a4wide} % nice big pages.

\usepackage[all]{xy} % Use this to easily draw dots-and-boxes diagrams on a matrix-oriented canvas
\usepackage{appendix}

\newcommand{\Z}{\mathbf{Z}}

\usepackage{amsthm}
\theoremstyle{plain}
\newtheorem{theorem}{Theorem}
\newtheorem{proposition}[theorem]{Proposition}
\newtheorem{lemma}[theorem]{Lemma}

\newtheorem{corollary}[theorem]{Corollary}
\theoremstyle{remark}
\newtheorem{remark}[theorem]{Remark}
\newtheorem{example}[theorem]{Example}
\begin{document}
\maketitle % puts the title in, and the author if there is one.
\begin{abstract}
We explain a highly efficient algorithm for playing the simplest type of
dots and boxes endgame optimally (by which we mean ``in such a way
so as to maximise the number of boxes that you take''). More precisely,
our algorithm applies to any endgame made up of long chains of any length
and long loops of even length (loops of odd length can show up in
generalisations of dots and boxes but they cannot occur in the original
version of the game).
The algorithm is sufficiently
simple that it can be learnt and used in over-the-board games by humans.
The types of endgames we solve come up commonly in
practice in well-played games on a $5\times5$ board and were in fact developed
by the authors in order to improve their over-the-board play.
\end{abstract}

\section{Introduction}
Dots and boxes is a two-person pencil-and-paper game, where players
take it in turns drawing lines on a (typically square) grid of dots.
If a player draws the the 4th line of a square, the player wins the
square and then makes another move. At the end of the game the person with
the most boxes wins. There are a couple of variants of the rules
in the literature; we are using the ``standard'' rules, as set out in Chapter~16
of~\cite{ww} and~\cite{berl}, which provide a very rich game
even on a board as small as $5\times 5$ (by which we mean 25 squares,
not 25 dots). These are also the rules implemented on various game sites
on the internet such as {\tt www.littlegolem.net}, {\tt www.jijbent.nl}
and {\tt www.yourturnmyturn.com}. In brief: you do not have to
complete a box, but if you do complete a box then you must make
another move. We follow Berlekamp in distinguishing between the idea of
a ``move'' (the act of drawing one line on the board) and a ``turn'' (the act
of drawing possibly several lines on the board, all but the last of
which completes at least one box).

Although perhaps initially counterintuitive, a crucial observation of
Berlekamp is that sometimes it is best not to make a box even if the
opportunity is available to you. Recall that a game is said to be played under
the Normal Play Rule if ``the player who completes the last legal
turn of the game wins''. Dots and boxes is not played under the Normal
Play Rule---at the end, box totals are counted up. However one can
introduce the game Nimdots,
which is dots and boxes but played under the Normal Play Rule.
It is noted in \cite{ww} that often a player will win the dots and boxes
game if and only if the same player wins the nimdots game. Furthermore
it is also not uncommon that the Nimdots optimal line of play (which
is trivial to compute in a given loony endgame position) will
win the dots and boxes game as well. This advice should nowadays be taken
with a pinch of salt in expert-level play, but is certainly often
true in endgame positions typically reached in games
between amateurs. 

As is explained in both~\cite{ww} and~\cite{berl}, there is a generalisation
of dots and boxes called strings and coins, played on what is basically
an arbitrary finite graph (the dictionary being that a box corresponds
to a vertex and an undrawn line to an edge), where players take turns in
cutting edges, and they claim isolated vertices. The analogous generalisation
of nimdots is called nimstring, and in this more general setting one
can of course see positions containing, for example, isolated
loops of any length (including odd lengths, or lengths less than~4; this
cannot happen in a
dots and boxes game). We will occasionally mention positions containing
loops of odd length, but our main results on endgames will be obtained under the
assumption that all loops have even length.

A lot has been written about nimdots and nimstring, which are far more
amenable to analysis than dots and boxes. Nimdots is an impartial game which
yields well to Sprague--Grundy (nim) theory. We do not go into this theory here
(see the references cited above for a very thorough analysis) because
the emphasis of this paper is not on winning nimdots but on winning
actual dots and boxes games, especially if they are close.
The sort of position that the
authors are interested in is something like Figure~\ref{exampleendgame}.
\begin{figure}[h!]
\centerline{
\xymatrix{
	{\bullet} \ar@{-}[r]	& {\bullet} \ar@{-}[d] \ar@{-}[r]	&{\bullet} \ar@{-}[r]	&{\bullet} \ar@{-}[d] \ar@{-}[r]	&{\bullet} \ar@{-}[r]	&{\bullet}\\
	{\bullet}	& {\bullet} \ar@{-}[d] &{\bullet}	&{\bullet} \ar@{-}[d] &{\bullet} \ar@{-}[d] \ar@{-}[r] &{\bullet}\\
	{\bullet} \ar@{-}[d] \ar@{-}[r]	& {\bullet} \ar@{-}[r]	&{\bullet} \ar@{-}[r] &{\bullet}	&{\bullet} \ar@{-}[d]	&{\bullet}\\
	{\bullet} & {\bullet}	\ar@{-}[d] \ar@{-}[r] &{\bullet}	\ar@{-}[r] &{\bullet} \ar@{-}[d] \ar@{-}[r] &{\bullet} \ar@{-}[r] &{\bullet} \ar@{-}[d]\\
	{\bullet} \ar@{-}[r] & {\bullet}	\ar@{-}[d] &{\bullet} \ar@{-}[d] &{\bullet} \ar@{-}[d] &{\bullet} &{\bullet} \ar@{-}[d]\\
	{\bullet}	& {\bullet}	&{\bullet} &{\bullet}	\ar@{-}[r] &{\bullet} \ar@{-}[r]	&{\bullet}
	}
}
\caption{A typical dots-and-boxes endgame}
\label{exampleendgame}
\end{figure}
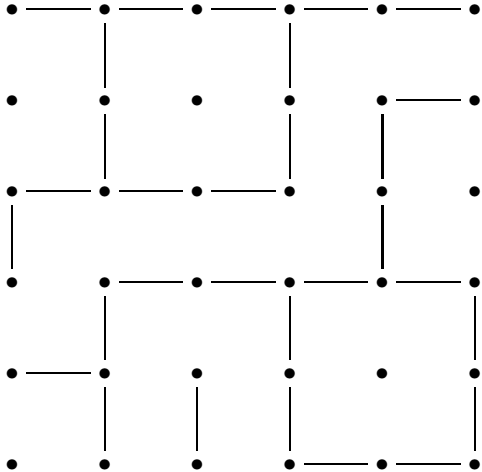
This is not a pathological position---this the sort of position
can easily occur as the endgame of a well-played game of
of dots and boxes.
Here 30 moves have been played. The assiduous reader of the references
above will know that player one has achieved his nimdots aim: the
board has two long chains (we refer the reader to the references
above for all basic definitions such as chains and loops) and hence
player 1 will win the nimdots game from this position. However player~2
has been smart and made some loops, it is
possible and relatively easy to prove that despite winning the nimdots
game, player~1 will actually lose the dots and boxes
game 13-12. This phenomenon is well-known to the authors of~\cite{ww}, who
observe on p569 that ``experts will need to know something about the rare
occasions when the Nimstring theory does not give the correct dots-and-boxes
winner'' and on p577 that ``Your best chances
at Dots-and-Boxes are likely to be found by the Nimstring strategy''.
These words have not aged well however. 
With the advent of games sites on the internet
such as {\tt jijbent.nl} and {\tt littlegolem.net} over the last few
years, the number of very strong dots and boxes players has skyrocketed,
and one of the first things that one discovers when playing games
against stronger players on these sites is that a nimdots player, even
one who has assiduously read the dots and boxes chapter of~\cite{ww}, stands
very little chance against someone who can count a dots and boxes endgame
correctly. This paper explains an optimal way of playing the
simplest possible dots and boxes endgames, composed only of isolated loops and
chains; even this simple situation is more subtle than one might imagine.
In fact there seems to be very little in the literature about
counting boxes, and more generally about the differences between
nimdots and dots and boxes. After some preliminary definitions we
will list everything that we are aware of.
Chapters~10 and~11 of~\cite{berl} provide
a good introduction to this material.

The \emph{value} of a game (played between $X$
and $Y$, with $X$ to play) is the number of future boxes that $Y$ will get
minus the number of future boxes $X$ will get, under best dots-and-boxes play.
For example, the value of an $n$-chain for $n\geq1$ is $n$, because $X$ will
have to open it and then $Y$ will take all of it.
Note that this is a very different notion to the nim-value of the game,
a notion that we will not consider at all in this paper
(the nim-value of any chain of length $n\geq3$ is zero). Values are considered
in Chapters~10 and~11 of~\cite{berl}, and also on
pages 574--575 of~\cite{ww}; this latter reference
states a formula for the value of a game comprising only
of long chains (without proof; the proof, which is not hard, is
given in Theorem~\ref{allchains} here).
Next, there is the paper~\cite{bs}, which makes
a beautiful analysis of some very topologically complicated endgames
that can occur on very large board. Finally there is Scott's MSc
thesis~\cite{scott},
available in the UC Berkeley library, which makes some more interesting
observations about values of certain loony dots and boxes endgames.

The contribution of this paper is to give a practical algorithm
useful for small-board play in a topologically simple endgame
such as Figure~\ref{exampleendgame}. In a sense our work is closest to pages
574--575 of~\cite{ww}, which we take much further because we
completely deal with the situation where the endgame contains
only chains and loops. The presence of loops complicates matters
immensely, but is crucial if one wants a practical algorithm because
in high-level play on a $5\times5$ board, the player who realises that
he will probably be losing the nimdots game (i.e., the chain battle)
will quickly turn their attention
to rigging the situation to winning the dots and boxes game regardless,
typically by making loops. Our philosophy when writing the paper was
to assume that the reader knows the theory as developed in~\cite{berl}
(which we only skim through), but then to give detailed proofs of the
results we need, including results which are stated without proof
in~\cite{berl} and~\cite{ww}.

\medskip

{\bf Acknowledgements.} KB would like to thank the player {\tt Carroll}
at {\tt littlegolem.net} for several helpful discussions about
values; in particular the correct formulation of Theorem~\ref{allloops} was
only made after one such discussion. MC would like to thank Astrid B\"onning, for sharing some original ideas. Part of this work
was initially done by MC for a Dutch school project, for which he would like to thank his teacher and supervisor, Bart van de Rotten.

\section{An overview of the problem considered in this paper.}
This paper considers dots and boxes positions that have reached
the ``loony endgame'' stage---that is, in which every available
move on the board not only gives away a box, but is a loony move
in the sense of~\cite{ww}. Recall that a loony move
is a move which gives your opponent the choice of who makes the last
move of the game. An example of a loony endgame would be a position
like Figure~\ref{nonsimpleloony}.
\begin{figure}[h!]
\centerline{
\xymatrix{
	{\bullet} \ar@{-}[r]	& {\bullet} \ar@{-}[r]	&{\bullet} \ar@{-}[d] \ar@{-}[r]	&{\bullet} \ar@{-}[r]	&{\bullet} \ar@{-}[d] \ar@{-}[r]	&{\bullet}\\
	{\bullet} \ar@{-}[r]	& {\bullet} &{\bullet} \ar@{-}[d]	&{\bullet} \ar@{-}[d] &{\bullet} \ar@{-}[d]	&{\bullet} \ar@{-}[d]\\
	{\bullet} \ar@{-}[d] \ar@{-}[r]	& {\bullet} \ar@{-}[r]	&{\bullet} \ar@{-}[d] &{\bullet} \ar@{-}[d]	&{\bullet} \ar@{-}[d]	&{\bullet} \ar@{-}[d]\\
	{\bullet} \ar@{-}[d] & {\bullet}	&{\bullet} \ar@{-}[d]	&{\bullet}	&{\bullet} \ar@{-}[d]	&{\bullet}\\
	{\bullet} \ar@{-}[d] & {\bullet} \ar@{-}[r]	&{\bullet} \ar@{-}[r]	&{\bullet} \ar@{-}[d] \ar@{-}[r] \ar @{} [dr] |*+{X}	&{\bullet} \ar @{} [dr] |*+{X} \ar@{-}[d] \ar@{-}[r]	&{\bullet} \ar@{-}[d]\\
	{\bullet} \ar@{-}[r]	& {\bullet} \ar@{-}[r]	&{\bullet} &{\bullet} \ar@{-}[r]	&{\bullet} \ar@{-}[r]	&{\bullet}
	}
}
\caption{A loony endgame}
\label{nonsimpleloony}
\end{figure}
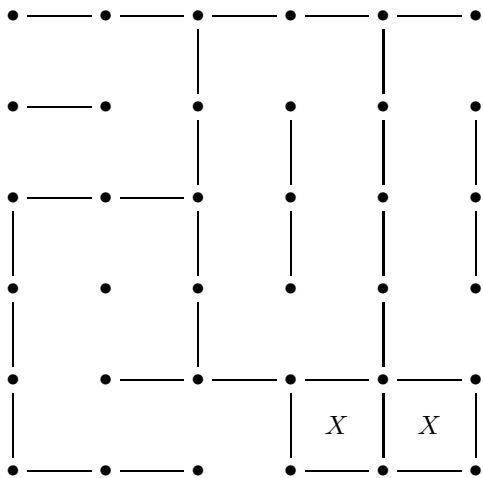
Here the only available moves are in chains of length at least 3 (a chain
of length at least~3 is called a ``long chain''), or loops (any loop
on a dots and boxes board has length at least~4 and is hence a ``long loop'').
Any move in a long chain gives your opponent the opportunity to play the
``all-but-two trick'' if they so desire, where they take all but two
of the boxes offered and then leave the last two by playing a so-called
``double-cross'' move, as indicated in Figures~\ref{loonymove} and~\ref{loonymove2}.
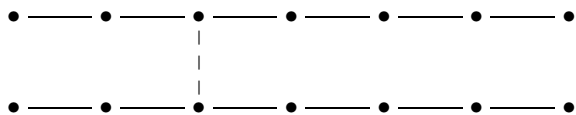
\begin{figure}[h!]
\centerline{
\xymatrix{
	{\bullet} \ar@{-}[r]		& {\bullet} \ar@{-}[r]		&{\bullet} \ar@{--}[d] \ar@{-}[r]		&{\bullet} \ar@{-}[r]		&{\bullet} \ar@{-}[r]		&{\bullet} \ar@{-}[r]		&{\bullet}	\\
	{\bullet} \ar@{-}[r]		& {\bullet} \ar@{-}[r]		&{\bullet} \ar@{-}[r]		&{\bullet} \ar@{-}[r]		&{\bullet} \ar@{-}[r]		&{\bullet} \ar@{-}[r]		&{\bullet}
	}
}
\caption{Player $X$ plays the vertical move opening a chain\ldots}
\label{loonymove}
\end{figure}
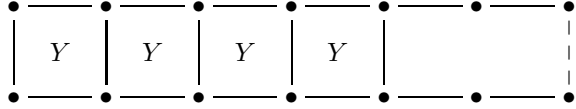
\begin{figure}[h!]
\centerline{
\xymatrix{
	\ar @{} [dr] |*+{Y}{\bullet}\ar@{-}[d]\ar@{-}[r]	& \ar @{} [dr] |*+{Y}{\bullet}\ar@{-}[d]\ar@{-}[r] & \ar @{} [dr] |*+{Y}{\bullet}\ar@{-}[d]\ar@{-}[r]	& \ar @{} [dr] |*+{Y}{\bullet}\ar@{-}[d]\ar@{-}[r]	& {\bullet}\ar@{-}[d]\ar@{-}[r]	& {\bullet}\ar@{-}[r]	&{\bullet}\ar@{--}[d]	\\
	{\bullet}\ar@{-}[r]	& {\bullet}\ar@{-}[r]	&{\bullet}\ar@{-}[r]	&{\bullet}\ar@{-}[r]	&{\bullet}\ar@{-}[r]	&{\bullet}\ar@{-}[r]	&{\bullet}
	}
}
\caption{Instead of taking all six boxes, player $Y$ takes four boxes but
leaves the last two for $X$, forcing
him to make another loony move elsewhere after taking his two
free boxes.}
\label{loonymove2}
\end{figure}

Similarly there is an ``all but four'' trick that applies if a player
opens a loop: their opponent can take all of the loop and play first
in the remainder of the game, or (as in Figure~16 on p554 of~\cite{ww})
they can leave four boxes and force
the opener to play first in the remainder of the game.

Our main results apply to what we call ``simple loony
endgames'', which are endgames where every connected component of the game
is either
an isolated loop or an isolated long chain. We also assume
that every loop is of even length (this is automatic for a dots and boxes
position, but may not hold for a general strings and coins
position\footnote{There is currently, as far as we know, still no known
method for \emph{efficiently} analysing a general loony strings and coins
position consisting of only isolated long loops and chains; the only method
we know is to do a brute force search down the game tree. The point of
this paper is to show how one can get away with much less if all loops
have even length.}). Such positions arise
commonly in high-level play. Figure~\ref{simpleloony} is an example which we will
work through carefully to highlight some issues. It is a game
between player~$X$ and player~$Y$, and $Y$ started.
One checks that 21 moves have been played, and no boxes have been taken,
so it is $X$'s turn. An expert will see instantly that~$X$ has lost,
but we want to talk about maximising the number of boxes they get.

\bigskip
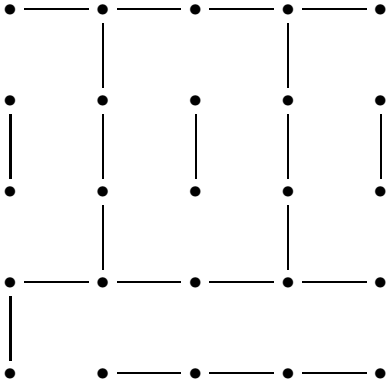
\begin{figure}[h!]
\centerline{
\xymatrix{
	{\bullet} \ar@{-}[r] & {\bullet} \ar@{-}[d] \ar@{-}[r]	&{\bullet} \ar@{-}[r]	&{\bullet} \ar@{-}[d] \ar@{-}[r]	&{\bullet}\\
% \ar@{-}[r]\\
	{\bullet} \ar@{-}[d]	& {\bullet} \ar@{-}[d] &{\bullet} \ar@{-}[d] &{\bullet} \ar@{-}[d] &{\bullet} \ar@{-}[d]\\
	{\bullet} & {\bullet} \ar@{-}[d] &{\bullet} &{\bullet} \ar@{-}[d] &{\bullet} \\
	{\bullet} \ar@{-}[r] \ar@{-}[d] & {\bullet} \ar@{-}[r]	&{\bullet} \ar@{-}[r]	&{\bullet} \ar@{-}[r] &{\bullet} \\
%	{\bullet} \ar@{-}[d] & {\bullet} \ar@{-}[r]	&{\bullet} \ar@{-}[r]	&{\bullet} \ar@{-}[r] & {\bullet}
	{\bullet}  & {\bullet} \ar@{-}[r]	&{\bullet} \ar@{-}[r]	&{\bullet} \ar@{-}[r] & {\bullet}
	}
}
\caption{A simple loony endgame}
\label{simpleloony}
\end{figure}
\bigskip

Our notation for such positions: we denote by $n$ a chain of length $n$,
and by $n_\ell$ a loop of length $n$, so we would call this
position $3+3+4+6_\ell$, or $2\cdot3+4+6_\ell$.
If this were a game of nimdots, then $X$ would open one of the
loony regions (a 3-chain, the 4-chain or the 6-loop)
and $Y$ would keep control (that is, if $X$ opens
a chain then $Y$ would play the all-but-two trick,
and if $X$ opens the loop then $Y$ would play the all-but-four trick).
In either case $X$
gets some boxes, but then has to open the next loony region. $Y$
may continue in this way and ensure that they have the last complete
legal turn of the game. Hence~$Y$ can win the nimdots game, but of course
this line results in a 10-6 win for $X$ in the dots and boxes game
($X$ gets two boxes for each chain and four for the loop).
Of course $Y$ has far better options. We shall see in this paper
that one of the correct ways to play this position (there are several)
is as follows:

\begin{enumerate}

\item $X$ opens a 3-chain and $Y$ takes all three boxes.

\item $Y$ opens a 3-chain and $X$ declines the offer, taking one box
and giving~$Y$ the other two.

\item $Y$ opens the 6-loop and $X$ takes all six boxes.

\item $X$ opens the 4-chain and $Y$ takes all four boxes.
\end{enumerate}

The final score is 10-8 to $Y$.

What is going through each of the players' heads during this exchange?
At each stage of this process above, \emph{both} players had a question to
answer. The person not in control (the one opening the loony
regions) had to play a loony move and so they had to decide which
component of the game to open.
And then the person in control had to decide whether
to keep control (sacrificing some boxes in the process), or whether to
grab everything on offer and lose control. \emph{How are these
decisions made optimally?} 

By a ``simple loony endgame'' we formally mean a loony endgame each of
whose components are long chains or long loops of even length -- such
as the position in Figure~\ref{simpleloony}.
In this paper we explain two algorithms. One of them, when given a simple
loony endgame, quickly tells the player whose turn it is (that is, the player
not in control), which loony region should be opened first.
The other, when given a simple loony endgame, quickly returns its value;
given this second algorithm it is then a simple matter (see Section~\ref{stayincontrol}) for the player
in control to decide whether to remain in control or not. Hence if both
players have access to both algorithms, they can play the endgame
optimally. By ``quickly'' here we mean either
``without traversing the entire game tree'',
or we could mean ``quickly enough to be of practical use in an over-the-board
game''. We could even give a more formal definition of what we have
here; if we are given the game as an input in a certain format, then
our algorithms will terminate in time $O(1)$ \emph{independent of
the number of components on the board}; in particular we are beating
the ``traversing the game tree'' approach hands down. To give an
example of what we are doing in this paper, imagine a position consisting
of a hundred 3-chains and a hundred 4-loops. The naive method
for computing whether to open a 3-chain or a 4-loop would involve
recursively computing the values of a position consisting of $a$ 3-chains
and $b$ 4-loops for all $0\leq a,b\leq 100$. However (once one has
understood the notions of value and controlled value, which are
about to be explained)
Theorem~\ref{theoremmultiple3chains} of our paper implies immediately
that the player whose turn it is should open a 3-chain,
and Corollary~\ref{vmultiple3chains} implies immediately that their
opponent should take all three boxes. Essentially no calculations at all are
required, and certainly no iteration over all subgames.

The algorithms are given in Section~\ref{sect:algos}, and the reader
only interested in playing dots and boxes endgames optimally can
just commit these algorithms to memory (or, even better, learn the
few underlying principles behind them), and ignore the proofs
that the algorithms are correct.

\section{Notation and basic examples.}

We will typically let $G$ denote a dots and boxes game (for
example $G$ could be ``the opening position in a $5\times 5$ game''
or ``an endgame with five 5-chains''. 
A \emph{loony endgame} is a game $G$ where every available
move is a loony move. A \emph{simple loony endgame} is a game
comprising only of disjoint long loops of even length and chains (for example
the game in figure~\ref{simpleloony}, but not the game in figure~\ref{nonsimpleloony}, as this
has a chain running into a loop).

We say that the \emph{value} $v(G)$ of the game~$G$ is the net score
for \emph{the player whose turn has just finished}
if the game is played optimally (in the sense that both players
are trying to maximise the number of boxes they obtain). We ignore
any boxes taken prior to reaching the position $G$; for example the
value of the empty game is zero (whatever the scores of the players
are at this point). It is natural
to look at the net gain of the second player rather than the first,
for this convention means that the value of a loony endgame is always
non-negative: if a player makes a loony move, then it is clear that
their opponent will not make an overall net loss if they play optimally,
because the sum of the values of their two options is non-negative
and hence they cannot both be negative.

We can easily extend this notion of value to define the \emph{value of a move}:
we write $v(G;m)$ to denote the net score for the player whose turn has just
finished if the move $m$ is played by his opponent, and then the resulting
game is played optimally by both players. If $Z$ is an isolated long loop
or chain in~$G$ then we denote by $v(G;Z)$ the value $v(G;m)$ for $m$ any
move in~$Z$ (all such moves are loony moves and have the same value).

Sometimes it is convenient to refer to a game as the collection of several
disjoint components. We may for example write $G=H+K$, where we mean
that $G$ is the game consisting of the disjoint union of games $H$ and $K$.
One case where this is useful is when we want to evaluate how the value of a
game changes when we slightly change one component. Say, for example, we are
interested in the game $G+n$, which consists of $G$ and a long chain of
length $n$. Its value $v(G+n)$ is then a function of $n$.

A notion that was introduced by Berlekamp and has proven to be very useful is
the \emph{controlled value} of a game. In fact we use this notion only
for simple loony endgames. For~$G$ a simple loony endgame, we first define
the \emph{fully controlled value} $fcv(G)$ to be the net score
for the player in control, if he keeps control for the entire game,
even to the extent of giving his opponent the last few boxes in the
connected region that is opened last. The controller hence loses
two boxes in each chain and four boxes in each loop. In particular
if~$G$ consists of long (by which we mean ``of length at least~3)
chains of lengths $c_1,c_2,\ldots$, and long (by which we mean ``of length
at least~4) loops of even lengths $\ell_1,\ell_2,\ldots$,
then $fcv(G)=\sum_i(c_i-4)+\sum_j(\ell_j-8)$.

The \emph{controlled value} $cv(G)$ of a simple loony endgame~$G$
is equal to the sum $fcv(G)+tb(G)$, where $tb(G)$, the \emph{terminal
bonus} of~$G$, is an integer calculated thus. The empty game has terminal
bonus zero. If $G$ is non-empty and has a chain of length at least~4,
or no loops, then $tb(G)=4$. If $G$ has loops but no chains,
then $tb(G)=8$. The remaining case is when $G$ has chains and loops,
but all chains are 3-chains; then we set $tb(G)=6$. The motivation
behind this definition is that the controlled value of a position
is the net gain for the player in control, assuming he remains
in control until the last (or occasionally the last-but-one) turn
of the game (we shall clarify this comment later).
Note in particular that whilst $v(G)$ is a subtle
invariant of~$G$, computing
$cv(G)$ is trivial for any simple loony endgame. It is easy
to check that $v(G)\equiv cv(G)$~mod~2 (as both are congruent
mod~2 to the total number of boxes in $G$). After we have developed
a little theory we prove (Lemma~\ref{lemmavcv}(a)) the assertion
on p86 of~\cite{berl},
namely that if $G$ is a simple loony endgame and $cv(G)\geq2$,
then $v(G)=cv(G)$.

For convenience, when talking about loony endgames we will refer to the player in control as the \emph{controller} and the other player as the
\emph{defender} from now on.

\section{The algorithms.}\label{sect:algos}

We briefly describe the two algorithms whose correctness we prove in this
paper. Both algorithms take as input a simple loony endgame.
The first algorithm computes its value (and is hence the one the controller
needs, to compute what to do after the defender opens a loony region
in a simple loony endgame). The second returns an optimal move
(and is hence the one used by the defender). 

\subsection{The algorithm to compute values.}

Let $G$ be a simple loony endgame.

\begin{itemize}
\item Is $G$ empty? If so, the value is zero. If not, continue.
\item Is $cv(G)\geq2$? If so, $v(G)=cv(G)$. If not, continue.
\item Are there any loops in~$G$? If so, continue.
If not, use Theorem~\ref{allchains} to determine $v(G)$.
\item Are there any 3-chains in~$G$? If so, continue. If not,
use Corollary~\ref{vno3chains} to determine~$v(G)$.
\item If $G$ has one 3-chain and every other component of $G$ is a loop, use Corollary~\ref{vloopsplus3} to determine~$v(G)$. If not, continue.
\item Does~$G$ have just one 3-chain? If so, use Corollary~\ref{vone3chainonebigchain} to compute $v(G)$, and if not, use Corollary~\ref{vmultiple3chains}.
\end{itemize}

\subsection{The algorithm giving optimal moves.}

Here we assume $G$ is non-empty. There will always be an optimal
move in~$G$ which is either opening the smallest loop or the smallest
chain (this is Corollary~\ref{smallestchainorloop}(a)).

\begin{itemize}
\item If $G$ has no chains, open the smallest loop. If $G$ has no loops,
open the smallest chain. If $G$ has both loops and chains, continue.
\item If $G$ has no 3-chains, open the smallest loop (Theorem~\ref{no3chains}).
Otherwise continue.
\item If $G$ has one 3-chain but every other component of~$G$ is a loop, use Theorem~\ref{loopsplus3} to
determine an optimal move. Otherwise continue.
\item If $G$ has one 3-chain, use Theorem~\ref{one3chainonebigchain}
to determine an optimal move. Otherwise use Theorem~\ref{theoremmultiple3chains}.
\end{itemize}

In fact applying the algorithms is relatively simple, but perhaps not quite as
simple as we have made things appear. Let us say
for example that we are trying to compute the value of a game~$G$
with loops and chains, but no 3-chains. Assume $cv(G)<2$.
Write $G=K+L$, where $L$ is all the 4-loops and $K$ is all the rest.
The algorithm above says that one should use Corollary~\ref{vno3chains}
to compute the value of~$G$, however looking at that corollary one
sees that to compute $v(G)$ one needs to compute $v(K)$, and
to compute $v(K)$ we need to apply the algorithm again! However,
one checks relatively easily that the number of times one needs to
``start again'' is at most~3 for any position, and that this
can only happen in some cases where $G$ has precisely one 3-chain
and at least one 4-loop. As we indicate in the paper, 
any time one needs to start again, the result in the paper
that one needs to compute the value of the subgame
comes strictly before the location of the order to start again,
so there is no way to get into an infinite loop.

To clarify, here is a worked example
of a worst-case scenario. Let $G$ be the game $3+4+{100 \cdot 4}_\ell+{100 \cdot 6_\ell}$
(a 3-chain, a 4-chain, a hundred 4-loops and a hundred 6-loops);
let us compute its value. Its controlled value is negative and odd.
Our first run through the value algorithm tells us to use
Corollary~\ref{vone3chainonebigchain}, where we see that to
compute $v(G)$ we need to know $v(G\backslash 3)$. Starting the value
algorithm again with $G\backslash 3$ (or just reading the remark after
Corollary~\ref{vone3chainonebigchain}) we are led to
Corollary~\ref{vno3chains}(d), where
we find that we need to compute
$v(G\backslash 3\backslash {100 \cdot 4_\ell})=v(4+{100 \cdot 6_\ell})$;
restarting for the second time (or just reading the comments in
Corollary~\ref{vno3chains}(d)) we find our way to Corollary~\ref{vno3chains}(c),
which tells us that $v(4+{100 \cdot 6_\ell})=4$. Hence $v(G\backslash 3)=4$.
This means (again by Corollary~\ref{vone3chainonebigchain}(e))
that we need to compute $v(G\backslash {100 \cdot 4_\ell})$, which by part (d)
is equal to~3 (note that we have already computed
$v(G\backslash {100 \cdot 4_\ell}\backslash 3)$), and hence $v(G)=3$ and we
are finished. This looks complicated but it is the worst-case scenario,
and often the calculations are much easier: the bottom line
is that computing the values of just
a few sub-positions has given us $v(G)$. An expert could have said
immediately that the value was either~1 or~3, but if you leave
your opponent the position $G+3$ and they open a 3-chain, you need
to know which possibility is the right one in order to decide whether
to double-cross or not.

\section{When to lose control.}\label{stayincontrol}

We briefly recall how the controller uses the algorithm which computes
the values of simple loony endgames, to answer the question of whether
to stay in control or not. This is standard stuff -- see for example
p82 of~\cite{berl}.

Consider a position $G+n$, that is, a game $G$ plus an
$n$-chain, with $n\geq3$. Let the players of this game be $X$ and~$Y$,
and let it be $X$'s move. Say $X$ opens
the $n$-chain. Then $Y$ takes all
but two of this chain (giving him a net score so far of $n-2$)
and then has to decide whether to take the last two
boxes and then to move first in~$G$, or whether to give the two
boxes to~$X$ and then move second in~$G$. Recall
that $v(G)$ is the net score for the second
player if the game is played optimally. If $Y$ keeps control
(and the game is then played optimally) then $Y$'s net gain will be
$n-2+(v(G)-2)$; if $Y$ relinquishes control, $Y$'s gain will be $n-2+(2-v(G))$.
The larger of these two numbers is $n-2+|v(G)-2|$, so
the critical question is
whether $v(G)\leq2$ or $v(G)\geq2$. If $v(G)>2$, $Y$ should keep control.
If $v(G)<2$, $Y$ should relinquish control, and if $v(G)=2$ then
it does not matter. We have proven
$$v(G+n;n)=n-2+|v(G)-2|$$
and may also deduce that if~$Y$ knows the value of~$G$ then they know
whether to keep control or not -- that is, they know how to play
the $n$-chain optimally after $X$ opens it.

Similarly, if $n_\ell$ denotes a loop of length $n\geq4$ then
$$v(G+n_\ell;n_\ell)=n-4+|v(G)-4|,$$ and faced with the
position $G+n_\ell$, with $n\geq4$, if $X$ opens the $n$-loop
then $Y$ keeps control if $v(G)>4$, loses control if $v(G)<4$,
and it does not matter what they do if $v(G)=4$.

In particular, given our two algorithms (one saying which loony move to play
in a simple loony endgame, the other computing the value of a simple
loony endgame), we know how to play the endgame perfectly.

Before we start on the proof of the correctness of our algorithms, we
develop some general theory. We start with a recap of the ``man in the
middle'' proof technique from Chapter~10 of~\cite{berl}.

\section{The Man in the Middle.}

The ``Man in the Middle'' is a technique explained on~p78 of~\cite{berl}.
In its simplest form it says that if the man in the middle (call him $M$)
is playing
two games of dots and boxes against two experts, and the starting positions
are the same, but~$M$ starts one game and does not start the
other, then~$M$ can guarantee a net score of zero by simply copying
moves from one game into the other -- the experts are then in practice
playing against each other. The crucial extension of this
idea is not to play the same game against the two experts, but to play
very closely-related games $G$ and $H$, typically identical apart from
in one simple region, and in this setting $M$ basically ``copies as
best they can'', where this has to be made precise when one of his opponents
plays in a region of one game which does not have an
exact counterpart in the
other. The outcome of such an argument will be a result
of the form $v(G)-v(H)\leq n$, or sometimes $|v(G)-v(H)|\leq n$,
when $G$ and $H$ are sufficiently similar. Note that we always
assume our experts are playing optimally, and so, for example,
if~$M$ opens a long chain against an expert then we may assume
that either they take all of the chain, or play the all-but-two
trick (other options, such as leaving three or more boxes, are
dominated by one of the two options above). 

We will now spell out one example in detail; the other results
in this section are proved using the same techniques.

\begin{example} If $J$ is an arbitrary game of dots and boxes,
then $|v(J+3)-v(J+4)|=1$. To prove this, set $G=J+3$ and $H=J+4$, and then let
the man in the middle, $M$, play both games, one as player~1 and the
other as player~2. Each time an opponent plays in one of the $J$'s, $M$
copies in the other~$J$. We now have to give a careful explanation
as to what algorithm $M$ will follow when the play moves out of~$J$.
When this happens, either an opponent has opened the 3-chain
(in which case $M$ opens the 4-chain in the other
game), or an opponent has opened the 4-chain (in which case $M$ opens
the 3-chain). $M$ now waits to see whether his opponent takes
all, or leaves two, and then copies the choice in the other game.
Note that at this point the games he's playing become the same, but
$M$ will either be a box down or a box up, depending on the choices
that his opponents made. We conclude that his net loss is no more than
one box, and hence $|v(J+3)-v(J+4)|\leq 1$. Now a parity argument
shows that in fact $|v(J+3)-v(J+4)|=1$.
\end{example}

Note finally that this argument does \emph{not} tell us whether
$v(J+3)>v(J+4)$ or $v(J+4)>v(J+3)$: both can happen, and distinguishing
between the two possibilities seems in general to be a very subtle issue.

An easy generalisation of the above argument gives the following
lemma.

\begin{lemma}\label{easyvals} Let $G$ be a dots and boxes game.

(a) If $1\leq m\leq n\leq 2$ or $3\leq m\leq n$ then $|v(G+m)-v(G+n)|\leq n-m$.

(b) If $4\leq m\leq n$ then $|v(G+m_\ell)-v(G+n_\ell)|\leq n-m$.
\end{lemma}
\begin{proof} 

(a) Mimic your opponents: the inequalities say that either both chains
are long or both are short, so this is possible.

(b) Mimic your opponents.
\end{proof}

\begin{corollary}\label{smallestchainorloop}
(a) If $3\leq m<n$ then in the game $G+m+n$, opening the $n$-chain
is never strictly better than opening the $m$-chain, and
similarly if $4\leq m<n$ then opening the $n$-loop in $G+m_\ell+n_\ell$
is never strictly better than opening the $m$-loop.

(b) In a loony endgame comprising entirely of long isolated chains
and long isolated loops, there will be an optimal play of the game
in which all the chains are opened in order (smallest first), as
are all the loops.
\end{corollary}

\begin{proof}

(a) Lemma~\ref{easyvals}(a) says $m-n\leq v(G+m)-v(G+n)\leq n-m$,
and hence $n+v(G+m)\geq m+v(G+n)$ and $n-v(G+m)\geq m-v(G+n)$.
Hence
$v(G+m+n;n)=\max\{n-4+v(G+m),n-v(G+m)\}\geq\max\{m-4+v(G+n),m-v(G+n)\}=v(G+m+n;m)$.
Hence your opening
the $n$-chain is always at least as good for your opponent as your
opening the $m$-chain. The same argument works for loops.

(b) Immediate from (a).
\end{proof}

As a consequence, we deduce that in a simple loony endgame, either
opening the smallest loop or the smallest chain is optimal! This
looks like a major simplification, but in fact distinguishing which
of the two possibilities is the optimal one is precisely the heart
of the matter.

Note also that as a consequence of this corollary, we may modify the
rules of dots and boxes by making it \emph{illegal} to open an $n$-chain if
there is an $m$-chain with $n>m\geq 3$; extra assumptions
like this simplify man in the middle arguments without changing
the values of any games, as we have just seen.

\section{Amalgamating loops and chains.}\label{amalsect}

Imagine you are winning a big chain battle in a dots and boxes game.
In amongst the game you can see two 5-chains, and you mentally
note that these will be worth one point each for you at the
end (as you will take three boxes and lose two, for each of them).
If someone were to offer to remove those two 5-chains and replace
them with a 6-chain (but without changing whose move it was, so you
were still winning the chain battle), then perhaps you would say
that you didn't mind much, because the two 5-chains are going to be worth two
boxes in total, which is what the 6-chain will bring in. More generally
you might be happy to swap an $a$-chain and a $b$-chain for an $(a+b-4)$-chain,
at least if $a,b\geq4$. Proposition~\ref{propamalgamate} below
proves that in fact
if $a,b\geq4$ then an $a$-chain and a $b$-chain are \emph{equivalent} to an
$(a+b-4)$-chain in huge generality.

\begin{proposition}\label{propamalgamate}Let $G$ be an arbitrary dots and boxes game,
and say $a,b\geq4$ and $c=a+b-4$. Then $v(G+a+b)=v(G+c)$.
\end{proposition}
\begin{proof} We use the man in the middle technique. The man
in the middle, $M$, plays two experts, playing the game $G+a+b$
with one and $G+c$ with the other, and starting in precisely one
of these games. 
We assume the experts play optimally. If we can show that $M$ always
at least breaks even on average
(that is, his net gain on one board is at least his net loss on the other,
always), regardless of which of the games he starts, then we have proved
the result. The work is in deciding how to play when one of $M$'s opponents
plays a move not in~$G$ and hence that cannot immediately be mirrored. The
easiest case is when an opponent opens the $a$-chain or the $b$-chain
(WLOG the $a$-chain). $M$ then doubledeals them (that is, he keeps
control), making a net gain of $a-4$
and leaving the games $G+b$ and $G+c$, whose values are known by
Proposition~\ref{easyvals}(b) to differ by at most $c-b=a-4$, so we
are finished in this case.

The fun starts when one of the experts opens the $c$-chain.
This move lets $M$ take $c-2$ boxes and then~$M$ has a decision
to make -- whether to keep control or not. If one of these decisions scores $v$
(not counting the $c-2$ boxes), then the other scores $-v$, so $v(G+c;c)$
is clearly at least $c-2$.
But we are assuming the experts are playing optimally,
so we may deduce $v(G+c)\geq c-2$. Again by Proposition~\ref{easyvals}(c)
we have $|v(G+c)-v(G+b)|\leq c-b$ (note that $c=b+(a-4)\geq b$)
and hence $v(G+b)\geq b-2\geq 2$. Hence if~$M$ opens the $a$-chain in $G+a+b$,
then we may assume that his opponent keeps control. Now $M$ has just made
a loss of $a-4$ but it is his move in $v(G+b)$ and he now opens the $b$-chain.
When his opponent takes the first $b-2$ boxes of this chain, we see
that $M$ has in total lost $a+b-6$ boxes in this game, and his opponent
has to decide whether or not to keep control in a game whose other component
is~$G$. On the other hand $M$ took $c-2=a+b-6$ boxes in the $G+c$ game,
and $M$ has to decide himself whether or not to keep control in this
game, so the positions have become identical and $M$ has a net score
of zero, and the proof is now complete.
\end{proof}

An easy induction on~$n$ gives

\begin{corollary}\label{amalchains} If $G$ is a game, if $n\geq1$,
if $c_i\geq4$ for $1\leq i\leq n$, and if $c=4+\sum_i(c_i-4)$, then
$v(G+c)=v(G+c_1+c_2+\ldots+c_n)$.
\end{corollary}
\qed

Now say $m$ is a move in the game $G$, say $c$ and the $c_i$ are
as in the previous corollary, and let $H_1=G+c$ and $H_2=G+c_1+c_2+\ldots+c_n$.
We can also regard $m$ as a move in $H_1$ and $H_2$.

\begin{lemma}\label{amalgoptimal}
The move $m$ is optimal in~$H_1$ if and only if it is optimal
in~$H_2$.
\end{lemma}
\begin{proof} Capturing free isolated boxes is always optimal,
so the result is obvious if $m$ is the capture of an isolated free box.
Indeed, we may assume that $G$ has no free boxes on offer.

Next say $m$ is a non-loony move in~$G$.
Then playing~$m$ has a cost~$s$ (the number of free boxes your opponent can take
after $m$) and, after playing~$m$ in~$G$ and then removing these boxes,
we are left with a game $G'$.
Set $H'_1=G'+c$ and $H'_2=G'+c_1+c_2+\ldots$. Then $m$ is optimal
in $H_1$ iff $v(H_1)=s-v(H'_1)$, and similarly for $H_2$. However
$v(H_1)=v(H_2)$ and $v(H'_1)=v(H'_2)$ by the preceding corollary,
and the result follows. 

Finally, assume~$m$ is a loony move, resulting in a handout of $s$ boxes
and a potential doubledeal involving $d\in\{2,4\}$ boxes (depending on whether
$m$ was in a chain or a loop). Let $G'$ denote the game obtained from~$G$
after all the available boxes are taken, and let $H'_1$ and $H'_2$
be as before. Then $m$ is optimal in $H_1$ iff $v(H_1)=s-|d-v(H'_1)|$,
and again by the previous corollary this is iff $v(H_2)=s-|d-v(H'_2)|$,
which is iff $m$ is optimal in $H_2$.
\end{proof}

Although less useful in practice on a small board, arguments
like the above can also be done for loops just as easily (with only
trivial modifications to the proofs). For example:

\begin{proposition} If $m,n\geq 8$ and $r=m+n-8$ then
$v(G+m_\ell+n_\ell)=v(G+r_\ell)$.
\end{proposition}
\qed
\begin{corollary} If $G$ is a game, if $n\geq1$, if $p_j\geq8$
for $1\leq j\leq n$ and if $p=8+\sum_j(p_j-8)$
then $v(G+p_\ell)=v(G+(p_1)_\ell+\cdots+(p_n)_\ell)$.
\end{corollary}
\qed
\begin{lemma} If $m$ is a move in $G$, then $m$ is optimal
in $G+p_\ell$ iff it is optimal in $G+(p_1)_\ell+\cdots+(p_n)_\ell$.
\end{lemma}
\qed

\section{The easiest examples: all loops or all chains.}

We finally begin the proof of the correctness of our algorithms.
In this section we deal with simple loony endgames consisting
either entirely of long chains, or entirely of long loops (of even length).
In this case, Corollary~\ref{smallestchainorloop} tells us
that an optimal move is opening the component with smallest size,
and our task is hence to compute the value of such a game.
We start with the case where $G$ is made up entirely of chains;
part (b) of the theorem below is stated without proof
in the section ``when is it best to lose control'' in Chapter~16
of~\cite{ww}, and in a sense our paper starts where they leave off.

\begin{theorem}\label{allchains}

(a) If a simple loony endgame $G=c_1+c_2+\ldots+c_n$
consists of $n\geq1$ disjoint chains of lengths $c_1$, $c_2,\ldots,c_n\geq4$,
then
$v(G)=cv(G)=4+\sum_{i=1}^n(c_i-4)\geq4$.

(b) If a simple loony endgame $G=c_1+c_2+\ldots+c_n$
is composed entirely of $n\geq1$ long chains, and if $c=cv(G)=4+\sum_i(c_i-4)$
is the controlled value of~$G$, then $v(G)=c$
if $c\geq1$, and $v(G)\in\{1,2\}$ with $v(G)\equiv c$~mod~2
if $c\leq0$.
\end{theorem}
\begin{remark} This theorem determines $v(G)$ uniquely for $G$ a simple
loony endgame with no loops, and furthermore (assuming we know the
controlled value of the game) it does it without having to
embark on a recursive procedure, computing values of various subgames.
Removing the need to ``go down the game tree'' in this way is
precisely the point of this paper.
\end{remark}
\begin{proof} (a) is an easy induction on $n$. For (b), induct on the
number of 3-chains, noting that if $G=3+H$ and $c'=cv(H)$ is the analogue
of the number $c$ for the game~$H$, then $c'=c+1$. If $c\geq 1$
then $c'\geq2$, so $v(H)=c'\geq2$ by the inductive hypothesis
and the controller should stay in control
giving $v(G)=v(H)-1=c$. On the other hand if $c\leq0$ then $c'\leq1$
so by the inductive hypothesis
$v(H)\in\{1,2\}\leq 2$ and the controller should lose control, giving
$v(G)=3-v(H)\in\{1,2\}$.
\end{proof}

We next deal with the case of simple loony endgames which are all loops.
This case is a little more complicated, because the 3-chain was the only
chain that needed to be dealt with separately above, whereas we must
deal with both 4-loops and 6-loops here as special cases. The following
theorem is the first time we need our assumption that all loops
are of even length (which will be true on a standard square dots
and boxes grid but may not be true in more generalised versions
of the game -- without this assumption then we would have to deal
with 5-loops and 7-loops as well).

Say $G$ is a disjoint union of long loops of even lengths $\ell_1$,
$\ell_2$, $\ell_3$,\ldots,$\ell_n$ with $n\geq1$. Set
$c=cv(G)=8+\sum_i(\ell_i-8)$
(and note that $c$ is even).
Let $f$ be the number of 4-loops in~$G$ and let $s$ be the number
of 6-loops.

\begin{theorem}\label{allloops} With notation as above,

(a) If $\ell_i\geq8$ for all $i$ then $v(G)=c$.

(b) If $c\geq2$ or $G$ is empty then $v(G)=c$.

(c) If $c\leq 0$ and $f=0$ and $G$ is non-empty then $v(G)\in\{2,4\}$ and
$v(G)\equiv c$~mod~4 (this congruence determines $v(G)$ uniquely).

(d) If $c\leq 0$ and $f\geq1$ then let $K$ denote $G$ minus all
the 4-loops (whose value is computable using (a)--(c) above).
If $v(K)\equiv2$~mod~4 then $v(G)=2$. Otherwise $v(G)\in\{0,4\}$
and $v(G)\equiv v(K)+4f$~mod~8.
\end{theorem}
\begin{proof}
Part (a) is an easy induction on number of loops, as in the previous
theorem. Part (b) follows by induction on $f+s$. The base case is (a),
and the inductive step proceeds as follows: if the length
of the smallest loop is~$t$ and $G=H+t_\ell$ then we know that opening
$t$ is optimal, and hence $v(G)=t-4+|v(H)-4|$. But $cv(H)=c-(t-8)\geq4$
so by the inductive hypothesis $v(H)=cv(H)\geq4$ and hence $v(G)=t-8+cv(H)=c$.

Part (c) is proved by induction on $s$. The case $s=0$ cannot occur, for
then~$G$ is a disjoint union of loops of length $\geq$ 8,
so its controlled value~$c$ must be at least 8.
The game cannot be one isolated 6-loop either, because we are
assuming $c\leq 0$. Hence we can write $G=6_\ell+H$ with $cv(H)=c+2\leq 2$.
Part (b) (if $cv(H)=2$) and the inductive hypothesis (if not)
implies $v(H)\in\{2,4\}\leq 4$ and
$v(H)\equiv c+2$~mod~4. Hence $v(G)=v(G;6_\ell)=2+|v(H)-4|=6-v(H)$
is congruent to $4-c$ and hence to~$c$ mod~4 (as $c$ is even).

Part (d) is similar -- an induction on $f$. Write $G=H+4_\ell$.
We have $cv(H)=4+c\leq 4$ so by (b) and the inductive hypothesis
we have $v(H)\in\{0,2,4\}$ and hence $v(G)=|4-v(H)|=4-v(H)$.
A case-by-case check now does the job; the hard work is
really in formulating the statement rather than checking its proof.
\end{proof}

\section{A general simple loony endgame: values and controlled values.}

Recall that the value of a simple loony endgame is a number which we
do not (at this stage in the argument)
know how to compute efficiently. However the controlled
value of such a game is very easy to compute. We are lucky then
in that the value and controlled value of a simple loony endgame
are by no means completely independent (for example they are congruent
mod~2, because both are congruent mod~2 to the total number of boxes
in the game). We give some more subtle relations
between these numbers here, which turn out to be crucial;
Lemma~\ref{lemmavcv}
and Corollary~\ref{corvcv} will be used again and again throughout
the rest of the paper.

\begin{lemma} Let $G$ be a simple loony endgame. Then $v(G)\geq cv(G)$.
\end{lemma}
\begin{remark} This is stated without proof on p84 of~\cite{berl};
we give a proof here for completeness.
\end{remark}
\begin{proof} 
We exhibit a strategy for the controller which, against best play,
gives the controller a net gain of $cv(G)$; this suffices.
Because the definition of $cv(G)$ has several
cases our argument must also have several cases. 

The case where $G$ is empty is clear. If $G$ is a non-empty
game and the controller simply remains in control until the last
move of the game, where he takes all the boxes, then this is
a play where the controller scores $fcv(G)+4$ or $fcv(G)+8$
depending on whether the last region opened was a loop or a chain.
We deduce that $v(G)\geq fcv(G)+4$ in all cases and that $v(G)\geq fcv(G)+8$
if $G$ is composed of a non-zero number of loops. This completes
the proof in all cases other than those with $tb(G)=6$.

So now let us assume that $G$ comprises at least one loop, at least
one 3-chain,
and that all chains are 3-chains, so $cv(G)=fcv(G)+6$.
Here is a strategy for the controller: keep control until the last
loop is opened, and then take all of it (and play optimally afterwards;
this is easy by Theorem~\ref{allchains}). We check that this guarantees
a score of at least $cv(G)$. If the last loop is opened on the last
move then the controller gets a score of $fcv(G)+8>cv(G)$.
If the last loop is opened earlier, and there are $n\geq1$ 3-chains
left, then after the controller has taken all of the loop
he has a net score of $fcv(G)+8+n\geq fcv(G)+9$ and is about to play a loony
move in a position comprising entirely of 3-chains, which has
value at most~3 by Theorem~\ref{allchains}(b), so his net gain
with this line is at least $fcv(G)+9-3=fcv(G)+6=cv(G)$.
\end{proof}
Part (a) of the next lemma is mentioned on p86 of~\cite{berl}.

\begin{lemma}\label{lemmavcv} Let $G$ be a simple loony endgame.

(a) If $cv(G)\geq2$ then $v(G)=cv(G)$.

(b) If $cv(G)<2$ then $0\leq v(G)\leq 4$, and $v(G)\equiv cv(G)$~mod~$2$.

(c) If $cv(G)<2$ and if $G$ has a 3-chain then $v(G)\leq v(G;3)\leq 3$.

\end{lemma}
\begin{proof}
(a) If there are no 3-chains or 4-loops or 6-loops in~$G$
then the fully controlled value of any non-empty subgame of~$G$ is always
at least~0, and so the controlled value is always at least~4. By the
preceding lemma the value of the subgame is always at least~4, and
hence the controller should initially remain in control
as the game progresses. However the value of the empty game is $0<2$,
and hence the controller should stay in control until the last move
of the game, whereupon he takes everything and scores $fcv(G)+4$ or $fcv(G)+8$
depending on whether the final move was in a loop or a chain. The defender's
optimal play is hence to ensure that the final move is in a chain, if
there are any chains, and we have proved $v(G)=cv(G)$ in this case.

The general case though is more complex. We prove the result by induction
on the number of components of~$G$. The strategy of the proof then is
for the defender to find a move in a component $C$ of $G$ such that
$cv(G\backslash C)\geq 2$ if $C$ is a chain and $cv(G\backslash C)\geq 4$
if $C$ is a loop. If the defender can always do this
then he is ``flying the plane''
(using the notation in~\cite{bs}) without crashing; this forces the controller
to stay in control at all times and this will prove the results. The
annoying added technicality
is that although it is easy to keep track of fully-controlled
values (the fully-controlled-value function is additive), the terminal
bonus function is not so well-behaved, so we need to check that
we can control this ``error term'' at all times.

One case that we have already dealt with is when there are no chains at all.
Then the result follows from Theorem~\ref{allloops}(b). Similarly
the case where there are no loops at all is Theorem~\ref{allchains}(b).

Another simple case is when $G$ contains a chain of length at least~4.
In this case the terminal bonus is always~4 and this will not change
if we remove some 3-chains, 4-loops and 6-loops. We know $v(G)\geq cv(G)$
so it suffices to find a strategy for the defender which results in a net
loss for him of only $cv(G)$. The defender can start by opening all
the 3-chains, 4-loops and 6-loops. For $cv(G)\geq2$ by assumption,
so if a 3-chain is opened then the controlled value goes up by~1
and hence the value of the new game is at least~3 and the controller
will stay in control. Similarly if the defender opens a 4-loop
or a 6-loop then the controlled value of the new game is at least~4
and hence the controller should stay in control. When all the 3-chains,
4-loops and 6-loops are opened the defender just opens all the loops
and then finally all the chains, and it is not hard to check that he
makes a net loss of $fcv(G)+4=cv(G)$.

The one remaining case is when there is at least one chain, at least
one loop, and all chains are 3-chains. In this case we see that the
defender can open all but one of the 3-chains and the controller must
stay in control because if $G=3+H$ with $H$ containing at least
one 3-chain then $cv(H)=1+cv(G)\geq3$, so the controller will stay
in control. When there is only
one 3-chain left, the defender starts opening all the not-very-long
loops (that is, 4-loops and 6-loops), and as above checks easily that the
controller has to stay
in control, until either there are no not-very-long loops left, or there
is only one loop left. The case of a 4-loop cannot occur as $cv(3+4_\ell)=1<2$.
The result is easily checked for $G=3+n_\ell$ with $n\geq6$.
Our final case is a 3-chain and
more than one very long loop. One checks that the defender can just
pick off the smallest very long loop, because the controlled value of
any game comprising a 3-chain and at least one very long loop is at least~5
so again the controller must stay in control.

(b) and (c): We prove both statements together, again by induction
on the number of components of~$G$. Of course the congruence mod~2
is automatic, as is $v(G)\leq v(G;3)$ in part (c) and $v(G)\geq0$ (as $G$
is loony); the issues are proving that $v(G)\leq 4$
and that furthermore $v(G;3)\leq 3$ if~$G$ has a 3-chain.
The case of $G$ empty is clear
so let us assume $G$ is non-empty. Then the assumption $cv(G)<2$ implies
that $G$ must contain either a 3-chain, a 4-loop or a 6-loop; we deal
with each case separately.

The easiest case is when $G=H+L$ with $L$ a 4-loop. Then $H$ must
be non-empty, so $tb(G)=tb(H)$
and hence $cv(H)=cv(G)+4\leq 5$, so $0\leq v(H)\leq 5$ by (a) and our inductive
hypothesis, and hence $v(G;L)=|v(H)-4|\leq 4$ and hence $v(G)\leq 4$.

The next case we consider is when $G=H+3$. Again~$H$ must
be non-empty. Removing the 3-chain
can sometimes make the terminal bonus increase (from~6 to~8) but
we certainly have $cv(H)\leq cv(G)+3\leq 4$ so, by the inductive
hypothesis, $v(H)\leq 4$
and hence $v(G)\leq v(G;3)=1+|v(H)-2|\leq 3$, and this proves (c).

The final case is when $G$ has no 4-loops or 3-chains, but has at
least one 6-loop. If $G=H+L$ with $L$ the 6-loop then $cv(G)<2$
implies $H$ is non-empty, so one
checks that $tb(G)=tb(H)$, and hence $cv(H)\leq 3$ so $v(H)\leq 4$
by (a) and our inductive hypothesis. Moreover $H$ has no 3-loops
or 4-chains, hence $v(H)\geq2$ (as any move by the defender sacrifices
at least 2 boxes before the decision of whether to keep control is
made), giving $v(G)\leq v(G;L)=2+|v(H)-4|\leq 4$.
\end{proof}
\begin{remark}\label{477}
Part (a) of the preceding lemma
becomes false if we drop the assumption that loops
have even length. For example if $G$ is the game $4+7_\ell+7_\ell$,
then $cv(G)=2$ but we claim $v(G)=4$. Indeed, $v(4+7_\ell)=3$ (open
the 7-loop), so $v(4+7_\ell+7_\ell;7_\ell)=4$ (the controller should
lose control because $v(4+7_\ell)<4)$. Similarly $v(7_\ell+7_\ell)=6$,
so $v(4+7_\ell+7_\ell;4)=6$ (the controller should this time keep control).
We deduce that $v(4+7_\ell+7_\ell)=4$.
\end{remark}
The preceding lemma showed how the controlled value influences the
value of a simple loony endgame. This corollary of it shows how the value
influences the controlled value.

\begin{corollary}\label{corvcv} Let $G$ be a simple loony endgame.

(a) If $v(G)\geq 5$ then $cv(G)=v(G)$.

(b) If $v(G)=4$ and if $G$ has a 3-chain then $cv(G)=v(G)$.

(c) If $G=H+n+3$ with $n\geq3$ and if $v(G;3)\geq 4$ then $cv(G)=v(G)=v(G;3)$.
\end{corollary}
\begin{proof} Parts (a) and (b) follow immediately from the preceding
lemma. 
Let us prove (c). Note first that $cv(G)\leq v(G)\leq v(G;3)$, so all
we need to show is $v(G;3)=cv(G)$.
We have $4\leq v(G;3)=1+|v(H+n)-2|$, and because
$v(H+n)\geq0$ we must have $v(H+n)\geq 5$. By (a) we have $cv(H+n)=v(H+n)$,
and hence $v(G;3)=v(H+n)-1=cv(H+n)-1=cv(G)$ and we are done. Note that it
is in the very last equality where we need the existence of the $n$-chain
(to stop the terminal bonus from changing).
\end{proof}

\section{Simple loony endgames with no 3-chains.}%\label{sectionno3chains}

We return to verifying the correctness of our two algorithms, this
time under the assumption that our simple loony endgame $G$ has no 3-chains. 
Our algorithm giving an optimal move for the defender is very
simple in this situation.

\begin{theorem}\label{no3chains} In a simple loony endgame with at least one
loop but no 3-chains, opening the smallest loop is an optimal move.
\end{theorem}

\begin{remark} Note that our proof \emph{does}
assume that the loops have even length. Indeed, if~$G$ is the strings-and-coins
game $G=4+4_\ell+4_\ell+7_\ell+7_\ell$ then opening the smallest loop
is not optimal. For $v(4+7_\ell+7_\ell)=4$ (see remark~\ref{477}),
hence $v(4+4_\ell+7_\ell+7_\ell)=0$ (as the value is non-negative but
opening the 4-loop has value~0) so $v(G;4_\ell)=4$; however
$v(4_\ell+4_\ell+7_\ell+7_\ell)$ is easily checked to be~2 and
hence $v(G;4)=2<4$ and opening the $4$-chain is better than
opening the 4-loop.
\end{remark}

Before we prove Theorem~\ref{no3chains}, we verify it in the special case where
the game has one 4-chain and no other chains at all. We then deduce
the result in general by using our technique of amalgamating
chains developed in Section~\ref{amalsect}.

\begin{proposition}\label{one4chain}
Let $G=K+L+4$ be a simple loony endgame consisting of a 4-chain, a loop~$L$,
and a (possibly empty) collection of loops~$K$ each of length at least
that of~$L$. Then opening~$L$ is an optimal move.
\end{proposition}
\begin{proof} Let us first deal with the case $v(G;L)>4$.
In this case we can even show
that $v(G;L)=cv(G)=v(G)$.
We will show this via a case-by-case check on the size of~$L$.
Recall that by definition of a simple loony endgame, $L$ and the
loops in~$K$ have even length, and hence $v(G;L)$ is even and hence
at least~6.

If $L$ has size~4, then $v(G;L)=|4-v(K+4)|\geq6$, and hence $v(K+4)\geq 10$,
so $cv(K+4)=v(K+4)\geq 10$ by Corollary~\ref{corvcv}(a),
and in particular $K$ is non-empty. We deduce $cv(G)=cv(K+4)-4\geq6$,
and now using Lemma~\ref{lemmavcv}(a)
we have $v(G)=cv(G)=cv(K+4)-4=v(K+4)-4=|4-v(K+4)|=v(G;L)$ in this case.

If instead $L$ is a 6-loop then we have $6\leq v(K+4+L;L)=2+|v(K+4)-4|$,
so $v(K+4)=0$ or $v(K+4)\geq8$. However
$v(K+4)=0$ is impossible, because $K+4$ is a non-empty loony
endgame with no 4-loops (recall $L$ is the smallest loop so
there are no 4-loops in~$K$). Hence $v(K+4)\geq8$ and the argument proceeds
just as in the case of a 4-loop above, the conclusion
being $v(G)=cv(G)=cv(K+4)-2=v(K+4)-2=2+|4-v(K+4)|=v(G;L)$.

The final case, under the $v(G;L)>4$ assumption, is
if $L$ is a loop of length $\ell\geq 8$. Then all loops have
length at least~8, so $fcv(G)\geq0$
and $fcv(K+4)\geq0$, hence $cv(G)$ and $cv(K+4)$ are both at least~4
and again we see $v(G)=cv(G)=\ell-8+cv(K+4)=\ell-8+v(K+4)=v(G;L)$.

Our conclusion so far is that if $v(G;L)>4$ then $v(G)=v(G;L)$.
We now treat the remaining possibilities.
Because $G$ is a simple loony endgame, the optimal move is
either in the smallest chain or the smallest loop by
Corollary~\ref{smallestchainorloop},
and hence $v(G)=\min\{v(G;4),v(G;L)\}$. Note also that $v(G;4)$ and
$v(G;L)$ are congruent mod~4 (because both are congruent mod~4 to the total
number of boxes in~$G$; this is because every component of~$G$
has even size and hence the value of any subgame of~$G$ is even).
Hence the only way that the Proposition can fail is if $v(G)=v(G;4)=0$
and $v(G;L)=4$. But this cannot happen as $v(G;4)=2+|2-v(K+L)|\geq 2$.
\end{proof}

\begin{proof}[Proof of Theorem~\ref{no3chains}]
%For $G$ an
%arbitrary game of dots and boxes, let us consider the
%function $f:\Z_{\geq3}\to\Z$ defined by $f(n)=v(G+n)$. By
%Lemma~\ref{easyvals}(a) we see that $|f(n+1)-f(n)|\leq 1$
%and hence by induction on~$n$ we see easily that for $n\geq4$
%we have $f(n)\leq f(4)+(n-4)$. 
Let us first consider a game $H$ of dots and boxes consisting
entirely of loops, with $m_\ell$ in $H$ the smallest loop.
Let us consider the function $g:\Z_{\geq4}\to\Z$ defined by
$g(n)=v(H+n;m_\ell)$ (the value of the game $H+n$ under the assumption
that the first player opens the $m$-loop). It follows
easily from Lemma~\ref{easyvals}(a) (applied with $G$ equal to $H$ minus
$m_\ell$) that $|g(n+1)-g(n)|=1$, and now an easy induction on $n$
shows that $g(n)\leq g(4)+n-4$ for $n\geq4$.

We have seen in Proposition~\ref{one4chain} that if $G=H+4$ with
$H$ comprising at least one loop, then an optimal move in~$G$ is
opening the shortest loop in~$H$. We now claim that the same
remains true in the game $H+n$ with $n\geq4$. For if opening the
smallest loop in~$H$ were not optimal, then opening the $n$-chain
must be optimal, and (with $g$ defined as above) we have
\begin{align*}
g(n)&=v(H+n;m_\ell)\\
&>v(H+n)=v(H+n;n)\\
&=n-2+|2-v(H)|=(n-4)+2+|2-v(H)|=n-4+v(H+4;4)\\
&\geq n-4+v(H+4;m_\ell)\mbox{ (by Proposition~\ref{one4chain})}\\
&=n-4+g(4)\geq g(n),
\end{align*}
a contradiction. Hence we deduce that if $H$ is non-empty and consists
entirely of isolated loops of even length, then an optimal
move in $H+n$ ($n\geq4$) is opening
the shortest loop.

The theorem now follows easily from our ``amalgamating chains'' technique.
Indeed Lemma~\ref{amalgoptimal} shows that if opening the shortest
loop is optimal in the game $H+n$ for any $n\geq4$,
then it is optimal in the game $H+C$
where $C$ is any non-empty collection of chains each of which
has length~4 or more.

\end{proof}

Theorem~\ref{no3chains} tells us in which order to open the components of
any simple loony endgame $G$ with no 3-chains: first all the loops are
opened, and then all the chains. But there still remains the issue of
when to lose control. We could use Theorem~\ref{no3chains}
to recursively compute the value of any simple loony
endgame with no 3-chains, by applying it to all the subgames that arise
and computing values of all of them in turn, and this would then give
us an algorithm
for playing such games optimally both as the controller and the
defender. However we want to avoid any recursion at all when running
our algorithms, which
fortunately we can do. Before we state our algorithm for computing
the value of such a simple loony endgame, here is an easy lemma which
we shall use in the proof and several more times later on.

\begin{lemma}\label{quadsandmod8}
If $w\in\Z_{\geq0}$, if $f\in\Z_{\geq0}$, if $w-4f\leq4$,
and if~$v$ is the result of iterating the function $x\mapsto|4-x|$, $f$ times,
on input $w$, then we have the following formula for~$v$: let $0\leq d\leq7$
be congruent to~$w$ mod 8; then $v=|4-d|$ if $f$ is odd,
and $v=4-|4-d|$ if $f$ is even.
\end{lemma}
\begin{proof} Induction on~$f$.
\end{proof}

We are now ready to compute the value of a simple loony endgame
with no 3-chains.

\begin{corollary}\label{vno3chains}
Say $G$ is a simple loony dots and boxes endgame, with no 3-chains.
Write $f$ for the number of 4-loops in $G$,
and set $c=cv(G)$.

(a) If $c\geq2$ or $G$ is empty then $v(G)=c$.

Assume from now on that $G$ is non-empty.

(b) If $c\leq1$ and $f=0$ and $c$ is odd then $v(G)=3$.

(c) If $c\leq1$ and $f=0$ and $c$ is even then $v(G)\in\{2,4\}$ and
$v(G)\equiv c$~mod~4 (this congruence determines $v(G)$ uniquely).

In the remaining case, $c\leq1$ and $f\geq1$, so there is a 4-loop.
Let $K$ denote~$G$ minus all the 4-loops; then $v(K)$ is computable by
(a)--(c) above.

(d) If $c\leq 1$ and $f\geq1$ and $K$ is as above, then let $0\leq d\leq 7$
be the unique integer in this range congruent to $v(K)$~mod~8. If $f$ is odd
then $v(G)=|4-d|$, and if $f$ is even then $v(G)=4-|4-d|$.
\end{corollary}

\begin{proof}
(a) This is just Lemma~\ref{lemmavcv}(a).

(b) and (c): these are proved simultaneously,
by induction on the number of 6-loops in $G$,
using the observations that if $G=H+6_\ell$
then $cv(H)=cv(G)+2\leq3$ and hence $v(H)\leq4$ by Lemma~\ref{lemmavcv},
and that $H$ has no 3-chains or 4-loops so $v(H)\geq2$.

(d) By Theorem~\ref{no3chains} we know that opening all the 4-loops
is an optimal line of play; $v(G)$ is hence computed from $v(K)$
by iterating the function $x\mapsto|x-4|$ $f$ times, so the result
follows from Lemma~\ref{quadsandmod8}.
\end{proof}

\section{Loops, one 3-chain, and no other chains.}

Based on the previous section one might hope that opening the smallest loop is
optimal in all simple loony endgames. Unfortunately, 3-chains
complicate matters immensely. For example, in the
simple loony endgame $G=3+6_\ell+6_\ell+6_\ell$ with three 6-loops and a 3-chain,
opening the smallest loop (one of the 6-loops) is not optimal: it
has a value of $v(G;6_\ell)=3>1=v(G;3)$. 
%
%\medskip
%\begin{figure}[h!]
%\centerline{
%\xymatrix{
%	{\bullet} \ar@{-}[d] \ar@{-}[r]	& {\bullet} \ar@{-}[r]	&{\bullet} \ar@{-}[d] \ar@{-}[r]	&{\bullet} \ar@{-}[r]	&{\bullet} \ar@{-}[d] \ar@{-}[r]	&{\bullet}\\
%	{\bullet} \ar@{-}[d] & {\bullet} \ar@{-}[d] &{\bullet} \ar@{-}[d]	&{\bullet} \ar@{-}[d] &{\bullet} \ar@{-}[d]	&{\bullet} \ar@{-}[d]\\
%	{\bullet} \ar@{-}[d] & {\bullet} &{\bullet} \ar@{-}[d] &{\bullet} &{\bullet} \ar@{-}[d]	&{\bullet}\\
%	{\bullet} \ar@{-}[d]  \ar@{-}[r] & {\bullet} \ar@{-}[r] &{\bullet} \ar@{-}[r]	&{\bullet} \ar@{-}[d] \ar @{}[dr] |*+{X} \ar@{-}[r] &{\bullet} \ar@{-}[d] \ar @{}[dr] |*+{X} \ar@{-}[r]	&{\bullet} \ar@{-}[d] \\
%	{\bullet} \ar@{-}[d] & {\bullet} \ar@{-}[r]	&{\bullet} &{\bullet} \ar@{-}[d] \ar @{}[dr] |*+{Y} \ar@{-}[r] &{\bullet} \ar @{}[dr] |*+{Y} \ar@{-}[d] \ar@{-}[r]	&{\bullet} \ar@{-}[d]\\
%	{\bullet} \ar@{-}[r] & {\bullet} \ar@{-}[r]	&{\bullet} \ar@{-}[r] &{\bullet} \ar@{-}[r] &{\bullet} \ar@{-}[r]	&{\bullet}
%	}
%}
%\caption{A simple loony endgame with loops and one 3-chain}
%\label{loopsplus3fig}
%\end{figure}
%\bigskip
The remainder of this paper is devoted to a proof of the correctness
of our algorithms in the cases where
our simple loony endgame has at least one 3-chain. We break the
argument into three cases. In this section we will fully analyse simple
loony endgames that contain exactly one 3-chain and no other chains at all.
In the next
section we deal with simple loony endgames that contain exactly one 3-chain and
also at least one chain of length 4 or more. Finally
Section~\ref{multiple3chains} handles the situation where there is
more than one 3-chain.

For the remainder of this section then, $G$ denotes a simple loony
endgame with one 3-chain and no other chains. If $G$ has no loops
at all then the game has value~3 and the only move is to open
the 3-chain, so let us for
the rest of this section assume that $G$ also contains at least one
loop. In general, sometimes
opening the smallest loop in $G$ is strictly better than opening the 3-chain,
sometimes it is strictly worse, and in many cases both moves are optimal.
We do not compute all of the optimal moves in any given position, we
are content with just finding one. We write $G=3+H$, where $H$ is
non-empty and composed entirely of loops. Note that Theorem~\ref{allloops}
tells us the value of~$H$ and all of its subgames. The next theorem
tells us an optimal move to play in~$G$, and although we may need
to apply Theorem~\ref{allloops} to compute the value of some subgames
of~$G$, it is clear from the statement of the theorem that we only
need to apply it twice.

\begin{theorem}\label{loopsplus3}
Say $G=H+3=K+L+3$ is a simple loony endgame,
where $H$ is non-empty and composed entirely
of disjoint loops, the smallest of which is~$L$.
The following strategy for playing~$G$ is optimal.
If $v(H)=2$ then open the 3-chain.
If $v(H)\not=2$, if $L$ is a 6-loop and and $v(K) = 2$, then open the 3-chain.
In all other cases, open $L$.
\end{theorem}
\begin{proof}
We know by Corollary~\ref{smallestchainorloop} that the optimal move in~$G$
is to open either~$L$ or the 3-chain. Note also that $cv(G)$ must be odd.
We first deal with the case $cv(G)\geq 3$.
Then $v(G)=cv(G)$ by Lemma~\ref{lemmavcv}(a);
however $cv(H)=cv(G)+3$ because $tb(H)=8$ whereas $tb(G)=6$;
so $v(H)=cv(H)\geq6$ and $v(G;3)=v(H)-1=v(G)+2$. This implies
that opening~$L$ is strictly better than opening the 3-chain, and our
task is to check that this is what the theorem predicts.
We have seen $v(H)>2$. Moreover, if $L$
were a 6-loop, then either $H=L$ so $v(H\backslash L)=0$, or
$H\backslash L$ is non-empty, so $tb(H\backslash L)=tb(H)=8$
and hence $cv(H\backslash L)=cv(H)+2\geq8$, implying $v(H\backslash L)\geq8$. 
Our theorem is hence correct in the case $cv(G)\geq3$.

Let us now assume $cv(G)\leq1$. We know $v(G)$ is odd and hence
$v(G)\geq1$. If $v(H)=2$ then $v(G;3)=1+|v(H)-2|=1\leq v(G)$,
and hence $v(G)=v(G;3)$ and opening the 3-chain is optimal, as predicted.

The next case we consider is when $cv(G)\leq1$, $v(H)\not=2$ and $L$
has length not equal to~6. First note that the length of~$L$ cannot
be~8 or more, because this would imply that all loops in~$H$ had
length 8 or more, hence $fcv(G)\geq-1$ and so $cv(G)\geq5$, a contradiction.
The only possibility is that~$L$ is a 4-loop; we then claim that
opening~$L$ is optimal. For if it were not then $v(G;3)<v(G;L)$, but
$v(G;3)=1+|v(H)-2|\geq 3$ and hence $v(G;L)=v(K+3+L;L)\geq5$,
implying $|v(K+3)-4|\geq5$, so $v(K+3)\geq9$, so $cv(K+3)\geq9$
by Corollary~\ref{corvcv}(a) and hence $cv(G)\geq5$, a contradiction.

The next case we have to check is when $cv(G)\leq 1$, $v(H)\not=2$,
$L=6_\ell$ is a 6-loop, and $v(K)=2$. Then $v(H)=v(K+L;L)=2+|4-v(K)|=4$, and
thus $v(G;3)=3$; however $v(G;L)=2+|4-v(K+3)|\geq3$ (as it is odd)
and thus $v(G;3)\leq v(G;L)$ and opening the 3-chain is optimal.

It suffices then to show that if $cv(G)\leq 1$, $v(H)\not=2$,
$L$ is a 6-loop and $v(K)\not=2$, then opening $L$ is optimal.
We do this by contradiction. If opening $L$ were not optimal
then $v(G;3)=v(G)<v(G;L)$. Now $cv(G)\leq1$ and $v(G)$ is odd,
so by Lemma~\ref{lemmavcv}(b) we must have $v(G)\in\{1,3\}$.
However $v(G;3)=1+|2-v(H)|>1$. Hence $v(G)=v(G;3)=3$, so $v(G;L)\geq5$.
Thus $5\leq2+|v(K+3)-4|$, so either $v(K+3)=1$ or $v(K+3)\geq7$.
If $v(K+3)\geq7$ then $cv(G)\geq 5$, which is a contradiction.
The case $v(K+3)=1$ cannot occur either; $K$ cannot be empty, the
smallest loop in $K$ has length at least~6 and hence opening it
costs at least~2, so the only way that $v(K+3)$ can be~1 is if
$v(K+3;3)=1$, but this implies $v(K)=2$.
\end{proof}
The previous theorem means that we now have an efficient algorithm for
computing an optimal move for any simple loony endgame with one 3-chain
and no other chains; we may use this result to compute the value of any such
game and hence verify the correctness of our second algorithm in the
case where $G$ has only one 3-chain and no other chains.
\begin{corollary}\label{vloopsplus3}
Say $G=3+H$ is a simple loony endgame, where $H$ is non-empty and composed
entirely of loops. Write $f$ for the number of 4-loops in $H$,
and set $c=cv(G)$.

(a) If $c\geq2$ then $v(G)=c$.

(b) If $c\leq1$ and $v(H)=2$ then $v(G)=1$ (note that $v(H)$ can be computed
using Theorem~\ref{allloops}).

(c) If $c\leq1$ and $v(H)\neq2$ and $f=0$ then $v(G)=3$.

(d) If $c\leq 1$ and $v(H)\neq2$ and $f\geq1$ then let $M$ denote $G$
minus all the 4-loops (and note that $v(M)$ can be computed using (a)--(c)
above). Let $0\leq d\leq 7$ be the unique integer in this range congruent
to $v(M)$ modulo~8. If $f$ is odd then $v(G)=|4-d|$, and if
$f$ is even then $v(G)=4-|4-d|$.
\end{corollary}

\begin{proof}
(a) This is just Lemma~\ref{lemmavcv}(a).

(b) If $v(H)=2$ then $v(G;3)=1$ and hence $v(G)=1$ because $v(G)$ is
non-negative, odd, and at most~1.

(c) The assumptions imply $v(G;3)\geq3$; there are no 4-loops, so if~$L$
is the smallest loop then $v(G;L)\geq2$, and $v(G;L)$ is odd so $v(G;L)\geq3$.
Hence $v(G)\geq3$. But $v(G)\leq3$ by Lemma~\ref{lemmavcv}(c).

(d) Write $M=N+3$, where $N$, possibly empty, is composed entirely
of loops of length~6 or~more. If $d\in\Z_{\geq0}$ we write $N+d\cdot 4_\ell$
for the game comprising the position~$N$ plus $n$ 4-loops, so $H=N+f\cdot 4_\ell$.
Now $H$ has no chains, so an optimal way to play~$H$ is to open all the loops
in order, starting
with the smallest; hence if $v(N+d\cdot 4_\ell)=2$ for some $1\leq d\leq f$,
then $v(H)=2$ contradicting our assumptions. By Theorem~\ref{loopsplus3}
we see that opening the 4-loop is optimal in $N+d\cdot 4_\ell+3=M+d\cdot 4_\ell$ for all
$1\leq d\leq f$. In particular we can compute $v(G)$ by starting with
$v(M)$ and then iterating the function $x\mapsto|4-x|$, $f$ times.
We know $v(G)\leq3$ by Lemma~\ref{lemmavcv}(c), and the result follows
by Lemma~\ref{quadsandmod8}.
\end{proof}

\section{Loops, very long chains, and one 3-chain}

In the previous section we analysed games consisting of some loops and a single 3-chain. In this section we will add long chains of length at least 4 to that situation. We begin by proving that our algorithm for predicting an optimal
move is correct in this situation.

\begin{theorem}\label{one3chainonebigchain}
If $G$ is a simple loony endgame that contains a single 3-chain and at least one chain of length $n \geq 4$, the following strategy is optimal:
If $cv(G) \leq 1$ and $G$ contains a 4-loop, write $G=H+3+4_\ell$ and open the 4-loop if $cv(H+3)=4$ or if $v(H+4_\ell) \in \left\{0,4\right\}$. In all other cases, open the 3-chain.
\end{theorem}
\begin{remark} Note that $v(H+4_\ell)$ can be computed using Corollary~\ref{vno3chains}.
\end{remark}
\begin{proof}
As usual, the proof breaks into a number of cases; in each case we verify
that the theorem gives us an optimal move in each case.

The first case we consider is the case $cv(G)\geq2$. Then
$cv(G\backslash 3)=cv(G)+1\geq3$, so by Lemma~\ref{lemmavcv}(a)
we have $v(G\backslash 3)=v(G)+1\geq3$, hence $v(G;3)=v(G\backslash 3)-1=v(G)$
and so opening the 3-chain is optimal.

The second case is when $cv(G)\leq 1$ but $G$ has no 4-loop.
By Lemma~\ref{lemmavcv}(c) we have $v(G;3)\leq 3$, but if~$L$ is
any loop in~$G$ then $L$ must have length at least~6, hence $v(G;L)\geq2$.
Because $v(G;3)$ and $v(G;L)$ must be congruent mod~2, we deduce
that $v(G;3)\leq v(G;L)$ for any loop, and hence again opening
the 3-chain is optimal.

So from now on we may assume $cv(G)\leq1$ and $G$ has a 4-loop;
we write $G=3+4_\ell+H$. Again by Lemma~\ref{lemmavcv}(c)
we have $v(G;3)\leq 3$.

The next case, a very easy case, is when $cv(H+3)=4$. By Lemma~\ref{lemmavcv}
we have $v(H+3)=4$, so $v(G;4_\ell)=0$ and opening the 4-loop must be optimal.

The last case we deal with is when $cv(H+3)\not=4$. This implies
that $v(H+3)\not=4$ (if $cv(H+3)\geq2$ use Lemma~\ref{lemmavcv}(a),
and if $cv(H+3)<2$ use Lemma~\ref{lemmavcv}(c)), and hence $v(G;4_\ell)\not=0$.
Next we claim $v(G;4_\ell)\leq 4$; for if $v(G;4_\ell)\geq5$ then $v(3+H)\geq9$,
so $cv(3+H)\geq9$ (Corollary~\ref{corvcv}), so $cv(G)\geq5$ contradicting
$cv(G)\leq1$. We deduce $1\leq v(G;4_\ell)\leq 4$. Our aim is to work out
which of $v(G;4_\ell)$ and $v(G;3)$ is the smaller, and recall that we know
that these numbers are congruent modulo~2. We have
$v(G;3)=1+|2-v(H+4_\ell)|$ and we already know that this is at most~3.
But we are now finished, because if $v(H+4_\ell)\in\{0,4\}$ then
$v(G;3)=3\geq v(G;4_\ell)$ and we open the 4-loop,
but if $v(H+4_\ell)\not\in\{0,4\}$ then $v(G;3)\in\{1,2\}$ so
$v(G;3)\leq v(G;4_\ell)$ and we open the 3-chain.
\end{proof}

Once again we can use this result to verify that our algorithm predicting
the value of such a game is correct.

\begin{corollary}\label{vone3chainonebigchain}
If $G$ is a simple loony endgame with exactly one 3-chain and at least one chain of length 4 or more, and $c=cv(G)$, we can compute its value $v(G)$ as follows:

(a) If $c\geq2$ then $v(G)=c$.

(b) If $c\leq1$ and $c$ is even, the value of $G$ is 2 except if $G$ has a 4-loop and $cv(G\backslash 4_\ell)=4$, in which case $v(G)=0$.

(c) If $c\leq1$ is odd and $v(G\backslash 3)=2$, then $v(G)=1$.

(d) If $c\leq1$ is odd, $v(G\backslash 3)\neq2$, and $G$ has
no 4-loops, then $v(G)=3$.

(e) If $c\leq1$ is odd and $v(G\backslash 3)\neq2$ and $G$ has $f\geq1$
4-loops, then let $K$ denote $G$ minus all the 4-loops (whose value can be computed using (a), (c) and (d)), and let $0\leq d\leq 7$ be congruent to~$v(K)$
mod~8. If $f$ is odd then $v(G)=|4-d|$, and if $f$ is even then $v(G)=4-|4-d|$.
\end{corollary}
\begin{remark}\label{rkone3chainonebigchain} Note that $v(G\backslash 3)$ can be computed
using Corollary~\ref{vno3chains}.
\end{remark}
\begin{proof}
(a) This is Lemma~\ref{lemmavcv}(a).

(b) We note that Lemma~\ref{lemmavcv}(c) implies that
$v(G)\in\left\{0,2\right\}$. Furthermore $v(G)=0$ iff $G$ has a 4-loop
and $v(G\backslash 4_\ell)=4$. By Corollary~\ref{corvcv}(b) and Lemma~\ref{lemmavcv}(a) this occurs iff $cv(G\backslash 4_\ell)=4$.

(c) This is clear; $v(G\backslash3)=2$ implies $v(G;3)=1$.

(d) $c$ is odd, and hence the value of any loop move is odd. If there
are no 4-loops, then the value of any loop move is at least~3. 
Now $v(G\backslash3)\not=2$ implies $v(G;3)\geq3$, and we conclude $v(G)\geq3$.
We conclude by using Lemma~\ref{lemmavcv}(c).

(e) The result will follow from Lemma~\ref{quadsandmod8}, once we have
proved that opening all the 4-loops in~$G$ is an optimal line of play.
Let us consider the position $K+d.4_\ell$, where $1\leq d\leq f$.
We prove by induction on~$d$ that opening the 4-loop is optimal,
and this suffices. 

There are two cases to consider.
The first is when $cv(K+d\cdot 4_\ell)\geq2$. In this
case, $v(K+d\cdot 4_\ell)=cv(K+d\cdot 4_\ell)$
by Lemma~\ref{lemmavcv}(a). Moreover,
$cv(K+(d-1)\cdot 4_\ell)=4+cv(K+d\cdot 4_\ell)\geq6$, and hence
$v(K+(d-1)\cdot 4_\ell)=cv(K+(d-1)\cdot 4_\ell)\geq6$,
so $v(K+d\cdot 4_\ell;4_\ell)=v(K+(d-1)\cdot 4_\ell)-4=v(K+d\cdot 4_\ell)$ and hence
opening a 4-loop is optimal in this situation. 

The other case is when $cv(K+d\cdot 4_\ell)\leq1$. Write $K=3+M$
and observe that for any $e\geq1$, the game $M+e\cdot 4_\ell$ has no 3-chains,
so by Theorem~\ref{no3chains} an optimal way to play it is to open
the 4-loop. We claim $v(M+d\cdot 4_\ell)\not=2$; for if $v(M+d\cdot 4_\ell)=2$
then $v(M+e\cdot 4_\ell)=2$ for all $e\geq d$ and in particular $v(G\backslash3)=2$,
contradicting our assumptions. We next note that $cv(K+d\cdot 4_\ell)\leq1$
implies $cv(M+d\cdot 4_\ell)\leq 2$ and hence $v(M+d\cdot 4_\ell)\leq4$ by
Lemma~\ref{lemmavcv}(a) and (b). Now $c$ is odd, so $v(M+d\cdot 4_\ell)$
is even, and hence $v(M+d\cdot 4_\ell)\in\{0,4\}$. 
So by Theorem~\ref{one3chainonebigchain} opening a 4-chain
is optimal in $3+M+d\cdot 4_\ell=K+d\cdot 4_\ell$, and we are finished.
\end{proof}

\section{Simple loony endgames with at least two 3-chains}\label{multiple3chains}

The last situation we need to analyse is when $G$ has at least two 3-chains
(and of course possibly other chains and loops).
Theorem~\ref{theoremmultiple3chains} describes an optimal strategy
for such a situation. This result concludes the proof that our
algorithm to compute an optimal move in a simple loony endgame is
correct.

\begin{theorem}\label{theoremmultiple3chains}
If $G$ is a simple loony endgame with at least two 3-chains, the following strategy is optimal: If there is a 4-loop in $G$, write $G=H+3+4_\ell$
and open the 4-loop if $cv(H+3)=4$ or $cv(H+4_\ell)=4$ or $cv(H)=4$.
In all other cases, open the 3-chain.
\end{theorem}
\begin{proof}
As usual, we break things up into cases, and verify that the theorem predicts
an optimal move in each case.

The first case we consider is when~$G$ has no 4-loop, and our task
is hence to prove that opening the 3-chain is optimal. If $v(G;3)\geq4$
then Corollary~\ref{corvcv}(c) does the job. If however $v(G;3)\leq3$
then opening the 3-chain must be optimal, because the value of opening
any loop is at least~2, and is congruent to $v(G;3)$~mod~2,
so must be at least~$v(G;3)$.

So we may now assume that~$G=H+3+4_\ell$ contains a 4-loop, and the
next three cases we consider will be the three cases where the
theorem tells us to open it. The easiest case
is when $cv(H+3)=4$; then $v(H+3)=4$ by Lemma~\ref{lemmavcv}(a)
and hence $v(G;4_\ell)=0$, so opening the 4-loop is optimal.

Before we continue with our case-by-case analysis,
we observe that if $v(G;4_\ell)\geq5$ then
$v(H+3)\geq9$, and by repeated uses of Lemma~\ref{lemmavcv}(a)
and Corollary~\ref{corvcv}(a) we may deduce $cv(H+3)\geq9$,
$cv(H)\geq10$ and $cv(H+4_\ell)\geq6$ (note $tb(H+4_\ell)\geq tb(H)$).
Hence $v(H+3)\geq9$, $v(H)\geq10$ and $v(H+4_\ell)\geq6$.

The next case we consider is when $cv(H+4_\ell)=4$. Then $v(H+4_\ell)=4$
and hence $v(G;3)=3$. Because $v(H+4_\ell)\leq 5$, the argument in the
previous paragraph shows that $v(G;4_\ell)\leq4$; moreover $v(G;4_\ell)$
is odd, and hence at most~3, and so opening the 4-loop is optimal,
as predicted.

The next case we consider is when $cv(H)=4$. The argument in the
last-but-one paragraph then shows that $v(G;4_\ell)\leq 4$
and hence $v(G;4_\ell)\leq3$.
Furthermore $v(H)=4$, so $v(H+4_\ell;4_\ell)=0$, thus $v(H+4_\ell)=0$.
We again deduce $v(G;3)=3$, so again opening the 4-loop in~$G$ is
optimal.

The final case we need to consider is when $G=3+4_\ell+H$,
$cv(H+3)\not=4$, $cv(H+4_\ell)\not=4$ and $cv(H)\not=4$; we
need to check in this case that opening the 3-chain is optimal.
If $v(G;3)\geq4$ then Corollary~\ref{corvcv}(c) gives the result.
If however $v(G;3)\leq 3$ then $v(G;3)\in\{1,2,3\}$ and we need
to rule out the case $v(G;4_\ell)<v(G;3)$. Because these numbers
are non-negative and congruent mod~2, there are only two possibilities:
the first that $v(G;4_\ell)=0<v(G;3)=2$, and the second that
$v(G;4_\ell)=1<v(G;3)=3$. The first is easily ruled out, as $v(G;4_\ell)=0$
implies $v(H+3)=4$ and hence $cv(H+3)=4$ by Corollary~\ref{corvcv}(b),
contradicting our assumptions. As for
the second, we see that $v(G;3)=3$ implies $v(H+4_\ell)\in\{0,4\}$.
The case $v(H+4_\ell)=4$ cannot happen because it implies $cv(H+4_\ell)=4$,
contradicting our assumptions. The final case
to deal with is $v(H+4_\ell)=0$; but this implies $v(H)=4$ and
hence $cv(H)=4$, again contradicting our assumptions.
The proof of the theorem is hence complete.
\end{proof}

Finally, we show how to use this result to compute the
value of the simple loony endgames in question; this will finish
the proof of correctness of our algorithm for computing
values of simple loony endgames.

\begin{corollary}\label{vmultiple3chains}
If $G$ is a simple loony endgame with at least two 3-chains, we can compute its value as follows:

(a) If $cv(G) \geq 2$ then $v(G)=cv(G)$.

(b) If $cv(G)<2$ and $cv(G)$ is even, the value of $G$ is 2 except if $G$ has a 4-loop and $cv(G\backslash 4_\ell)=4$, in which case $v(G)=0$.

(c) If $cv(G)<2$ and $cv(G)$ is odd then $v(G)=1$.
\end{corollary}

\begin{proof}
(a) This is Lemma~\ref{lemmavcv}(a).

(b) We know $v(G)$ is even, so Lemma~\ref{lemmavcv}(b) and (c)
imply $v(G)\in\{0,2\}$. If $G$ has no 4-loop then $v(G)=0$ is
impossible, and hence $v(G)=2$ as claimed. If $G=H+3+4_\ell$ has a 4-loop,
then we want to use Theorem~\ref{theoremmultiple3chains} to find an
optimal move. Note first that $cv(H+4_\ell)$ and $cv(H)$ are odd,
so they cannot be~4, and so the only question we need to consider
is whether $cv(G\backslash 4_\ell)=cv(H+3)=4$ or not. If $cv(H+3)=4$ 
then $v(G\backslash 4_\ell)=4$, so $v(G;4_\ell)=0$ and hence $v(G)=0$
as claimed. If however $cv(H+3)\not=4$ then Theorem~\ref{theoremmultiple3chains}
tells us that opening a 3-chain is optimal, and hence $v(G)>0$, giving
$v(G)=2$ as the only possibility.

(c) This time $v(G)$ is odd, and Lemma~\ref{lemmavcv}(b) and (c)
then imply $v(G)\in\{1,3\}$. If $v(G;3)=1$ then $v(G)=1$ which is
what we want. If $v(G;3)\geq5$ then Corollary~\ref{corvcv}(c)
implies $cv(G)\geq5$, a contradiction. The only possibility
left is $v(G;3)=3$, from which we need to deduce $v(G)=1$.

Now $v(G;3)=3$ implies $v(G\backslash3)\in\{0,4\}$. However $v(G\backslash3)=4$
is impossible, as Corollary~\ref{corvcv}(b) would then imply
$cv(G\backslash3)=4$ which would imply $cv(G)=3$, a contradiction.

We deduce $v(G\backslash3)=0$; however $G\backslash3$ is non-empty
(as it contains a 3-chain), and hence $G\backslash3$ must contain
a 4-loop. Write $G=3+4_\ell+H$, so $v(4_\ell+H)=0$ and hence $v(H)=4$.
But this implies $cv(H)=4$ (Corollary~\ref{corvcv}(b) again),
so $cv(H+3)=3$, thus $v(H+3)=3$ and hence $v(G;4_\ell)=1$, proving
that $v(G)=1$.
\end{proof}

%\appendix
%\appendixpage
%\bibliographystyle{plain} %Bibliography style file X.bst
%\bibliography{dots} % Bibliography database file Y.bib

\end{document}